\documentclass{article}
\usepackage{latexsym, a4wide}
\usepackage{amsmath, rotating, color, pdfsync}
\usepackage{amsfonts,amssymb, amsthm, mathrsfs}
\usepackage{bbm}
\usepackage{wrapfig}
\usepackage{amsmath}
\usepackage{natbib}
\usepackage[intlimits]{empheq}
\usepackage{multicol}
\usepackage{amsthm}
\usepackage[hidelinks]{hyperref}
\usepackage{cancel}

\usepackage{xspace}

\newcommand{\species}[3]{%
	\newcommand{#1}{\gdef#1{\textit{#3}\xspace}\textit{#2}\xspace}}

\species{\ecoli}{Escherichia coli}{E.~coli}

\newcommand\unnumberedfootnote[1]{ %
        \let\temp=\thefootnote %
        \renewcommand{\thefootnote}{}%
        \footnote{#1}%
        \let\thefootnote=\temp%
        \addtocounter{footnote}{-1}}

\newtheorem{theorem}{Theorem}
\newtheorem{proposition}{Proposition}[section]
\newtheorem{lemma}[proposition]{Lemma}
\newtheorem{corollary}[proposition]{Corollary}
\newtheorem{definition}[proposition]{Definition}
\theoremstyle{definition}
\newtheorem{remark}[proposition]{Remark}

\usepackage{yfonts, eufrak}
\usepackage[normalem]{ulem}
\numberwithin{equation}{section}

\DeclareMathAlphabet{\mathpzc}{OT1}{pzc}{m}{it}

\begin{document}
\title{\LARGE Modifiers of mutation rate in selectively fluctuating
  environments}

\author{\sc Franz Baumdicker, Elisabeth Huss, Peter Pfaffelhuber \\[2ex] University of Freiburg\\[2ex]}

\date{\today\unnumberedfootnote{\emph{AMS 2010 subject
      classification.} {\tt 92D15} (Primary) {\tt 60F17} (Secondary).}
  \unnumberedfootnote{\emph{Keywords and phrases.}  Fleming-Viot
    process, fluctuationg selection, modifier theory, microbial evolution, second order evolution, fixation probability} }

\maketitle

\begin{abstract}
  We study a mutation-selection model with a fluctuating
  environment. More precisely, individuals in a large population are
  assumed to have a modifier locus determining the mutation rate
  $u \in [0,\vartheta]$ at a second locus with types $v\in [0,1]$. In
  addition, the environment fluctuates, meaning that individual types
  change their fitness at some high rate. Fitness only depends on the
  type of the second locus. We obtain general limit results for
  the evolution of the allele frequency distribution for rapidly
  fluctuating environments.
  As an application, we make use of the resulting Fleming-Viot process and compute
  the fixation probabilities for higher mutation rates
  in the special case of two bi-allelic loci in the limit of small fitness differences at the second locus.
\end{abstract}

\section{Introduction}
Mutation is inarguably one of the fundamental forces behind evolution.
Mutations are DNA copying errors that result in the creation of new
alleles and thus drive genetic diversity within the population.  The
study of the rates at which mutations occur is therefore of great
interest.
It has been noticed early on that most mutations are deleterious
(e.g.~\citealp{Fisher1930}).  Mutating too often would hence most
likely cause an individual to be at a disadvantage relative to ones
that rarely mutate.  This seems fitting in an environment where no
change happens that could impact the fitness of these individuals.
Individuals that hardly ever mutate and are well-adapted to this fixed
environment are favored by selection, while those that produce too
many mutations which mostly do not bring about any improvements fail
to establish themselves within that population.  The result is then a
population with a relatively low mutation rate.  In population genetic
models a common assumption is thus a constant and often quite low
mutation rate.

However, this picture changes if the population is forced to adapt to
a moving fitness optimum. As an example, host-parasite interactions
can result in an evolutionary arms race both in eukaryotes (see e.g.\
\citealp{Davies1989}) and in prokaryotes (see e.g.\ \citealp{Koonin2017} and \citealp{Pal2007}).
More generally, if external influences change the environment in a way
that well-adapted individuals that had previously enjoyed the
preferential treatment by selection are faced with a decrease in
fitness, higher mutation rates might be beneficial.  Indeed, an
increase in the number of individuals with high mutation rates, often
called \emph{mutators}, has been observed in many experiments, where
bacterial populations are exposed to new environments forcing them to
quickly create better-adapted individuals~\citep{Denamur2006}.  One of
the earliest works that deals with this subject is
\cite{Sturtevant1937}.  This essay discusses the fact that mutation
rates can differ even within taxa and that genes affecting the
mutation rate succumb to selection.  As soon as adaptation is reached
and the environment does not change again,
there are no benefits in having higher mutation rates and selection
will again favor lower mutation rates \citep{Wielgoss2013a}. The study
of the rise and fall in frequency of these mutators and their role in
adaptive evolution has been gaining more and more attention over the
years especially for microbial evolution, see
e.g.~\cite{Tenaillon2001} for a review.

We study the evolution of mutation rates using modifier theory, where
an additional neutral modifier locus determines the mutation rate at a
second locus.  Modifier theory has been used to show that mutation
rates are reduced by indirect selection in constant
environments~\citep{Karlin1974,Liberman1986} and can be decreased or
increased in a random environment depending on the model parameters,
including the mean fitness differences between genotypes and the
variance and autocorrelation of the environment~\citep{Gillespie1981}.
Results on the evolution of modifier loci under fluctuating selection
strongly depend on the choice of model parameters.  Important
parameters are the speed and shape at which environmental changes are
triggered and the direction and strength of selection.  Mutators
increase in frequency hitchhiking beneficial
mutations~\citep{Johnson1999} and are indirectly selected against as
deleterious mutations accumulate faster in strains with higher
mutation rates~\citep{Dawson1998}.  Constant~\citep{Kessler1998},
moving~\citep{Tanaka2003} and periodically
fluctuating~\citep{Ishii1989,Travis2002} fitness landscapes have been
used to study the evolution of mutation rates and mutator frequencies
within bacterial populations.  At which speed and whether the
environment switches periodically or randomly affects not only
mutation~\citep{Ishii1989}, but also recombination and
migration~\citep{Carja2014} as well as phenotypic switching
rates~\citep{Hufton2016}.

In our work, we use the versatile framework of Fleming-Viot processes
to present a model describing the phenomenon of mutation modifiers
properly in rapidly fluctuating environments. We stick to a
prokaryotic (haploid) population evolving under mutation and
selection, but without recombination. In Section~\ref{sec_2}, we will
first derive a bivariate process that describes the mutation rate and
type space in the first variable and the fitness of the type in the
second variable which will act according to a fluctuating environment.
The process is defined as a solution to a well-posed martingale
problem and is called the \emph{Fleming-Viot process with mutation
  modifier and fluctuating selection}.  Our first result
(Theorem~\ref{T1}) is the convergence of this process to a unique
limit in the case of a fast fluctuating environment.  To show how the
results can be applied we continue in Section~\ref{sec_3} with a
special 2-type case where only two mutation rates and two types at the
second locus exist (Theorem~\ref{T2}).  We compute the fixation
probability of the high mutating type depending on the two mutation
rates in Theorem~\ref{T3}.

\begin{remark}[Notation]\label{rem:notation}
  We set $I:=[0,1]$. For a complete and separable metric space
  $(E,r)$, we denote by $\mathcal M(E)$ the space of measurable, by
  $\mathcal B(E)$ the space of bounded, measurable, by
  $\mathcal C_b(E)$ the space of bounded, continuous real-valued
  functions on $E$ (equipped with convergence of uniform on compacta)
  and for $L>0$ by $\mathcal C_L(E)$ the space of bounded, real-valued
  functions with Lipschitz constant $L$.
  For $\nu \in \mathcal P(E)$ -- the space of probability measures on
  $E$, equipped with the topology of weak convergence -- and
  $f \in \mathcal M(E)$, we write
  $\langle \nu, f\rangle := \int f(u)\nu(du)$, if the right hand side
  exists. We denote weak convergence by $\Rightarrow$. Note that this
  convergence relies on a topology on the underlying space.  More
  specifically, we rely on Skorohod convergence in path space.

  Below, we will be dealing with strongly continuous contraction
  semigroups. Recall that for some Markov process
  $X = (X_t)_{t\geq 0}$ with (locally compact and separable) state
  space $E$, the family of operators $(S_t)_{t\geq 0}$ given through
  $S_tf(x) := \mathbb E_x[f(X_t)]$ for $f\in\mathcal C_b(E)$ generates
  a semi-group (i.e.\ $S_t S_s = S_{t+s}$ by the Chapman-Kolmogorov
  equalities). It is a contraction since
  $||S_tf|| = \sup \mathbb E_x[f(X_t)] \leq ||f||$ and positive since
  $S_tf(x) = \mathbb E_x[f(X_t)]\geq 0$ for $f\geq 0$. In addition,
  such a semigroup has a conservative generator since $S_t1 = 1$ and
  is called strongly continuous if
  $S_tf(x) = \mathbb E_x[f(X_t)] \xrightarrow{t\to 0} f(x)$.  Also,
  recall that a positive, strongly continuous contraction semigroup
  with conservative generator and $S_tf \in \mathcal C(E)$ for all
  $t\geq 0, f\in\mathcal C(E)$ is called a Feller semigroup
  \citep{Kallenberg2002}.  Reversely, if $E$ is locally compact and
  separable, a Feller semigroup corresponds to a strong Markov process
  with sample paths in $\mathcal D_E([0,\infty))$; see
  \cite{EthierKurtz86}, Section~4.3. Such processes are therefore also
  called Feller processes.
\end{remark}

\section[A Fleming-Viot system with mutation modifier]{A Fleming-Viot
  system with mutation modifier and fast fluctuating selection}
  \label{sec_2}
Let us give some interpretation before we formally define the
generator of the sequence of Markov processes which we consider. We
will derive a Markov process $(X,Z)$ (more precisely we derive a
sequence of such processes and a limiting process) with state space
$S:=\mathcal P([0,\vartheta] \times I) \times \mathcal C_L(I)$ for
some $L>0$. For a sample $(u,v)$ from
$X_t \in \mathcal P([0,\vartheta] \times I)$ at time $t$, the first
coordinate, $u$, denotes the allele at the first locus (which we call
$A$-locus), whereas $v$ is the allele at the second
(\text{$B$-})locus. Here, $u \in [0,\vartheta]$ equals the mutation
rate of the sampled individual at the $B$-locus. Upon a mutation, the
allele at the $B$-locus is drawn from $\beta(v,.)$ (a transition
kernel on $I$). Selection acts on the $B$-locus according to some
fitness function $Z_t \in \mathcal C_L(I)$ at time $t$, which is
subject to fluctuations. The fitness function $Z_t$ changes along a
Poisson process to independent draws from
$\nu \in \mathcal P(\mathcal C_L(I))$.  We require that
$\mathbb E_\nu[Z(v)]=0$ for all $v\in I$, i.e.\ on average, no allele
at the $B$-locus has a fitness advantage.

~~

\noindent
We collect all assumptions and some notation in the following remark.

\begin{remark}[Assumption, state space and notation]\mbox{}
  \begin{enumerate}
  \item Let
    \begin{align*}
      \vartheta \geq 0, & \qquad \text{(maximal mutation rate at $B$-locus)},
      \\ \sigma\geq 0, &  \qquad  \text{(selection intensity)},
      \\ L \geq 0, &  \qquad  \text{(Lipshitz constant for fitness function)},
      \\ \gamma>0, &  \qquad \text{(rate of environmental change)},
      \\ \nu \in \mathcal P(\mathcal C_L(I)), & \qquad \text{(distribution of random fitness)},
    \end{align*}
    and $\beta$ a transition kernel from $I$ to $I$ (mutation kernel
    at the $B$-locus), such that $u\mapsto \beta(u,.)$ is
    continuous. Throughout, we assume that
    \begin{align}\label{eq:Z0}
      \mathbb E_\nu[Z(v)]=0, \qquad v\in I.
    \end{align}
  \item The state space of the Markov process in the next definition
    will be
    $S:=\mathcal P([0,\vartheta] \times I) \times \mathcal
    C_L(I)$. This space is equipped with the product topology, where
    $\mathcal C_L(I)$ is equipped with the topology of uniform
    convergence, and $\mathcal P([0,\vartheta] \times I)$ is equipped
    with the topology of weak convergence. Note that $S$ is locally
    compact.
  \item For $(u,v) \in [0,\vartheta] \times I$, we say that $u$ is
    the allele at the $A$-locus and $v$ is the allele at the
    $B$-locus. Denote by
    $\pi_A: [0,\vartheta] \times I \to [0, \vartheta]$ and
    $\pi_B: [0,\vartheta] \times I \to I$ the projections on
    the first and second coordinate, i.e.\ the $A$- and $B$-locus,
    respectively. More generally, for $k=1,...,n$, $\pi_{k,A}$
    ($\pi_{k,B}$) is the projection of $([0,\vartheta] \times I)^n$ to
    the $k$th entry at the $A$-locus ($B$-locus).
  \item For a transition kernel $\beta$ from $I$ to $I$ and
    $\phi \in \mathcal C(([0,\vartheta] \times I)^n)$, we set, for
    $u\in [0,\vartheta]^n$,
    \begin{align}\notag \beta_{k,B}\phi(u, v_1,...,v_n) := \int
      \beta(v_k, dv') \phi(u, v_1,...,v_{k-1}, v', v_{k+1},...,v_n)
    \end{align}
  \item For $z \in \mathcal C_L(I)$ and $v \in I^n$, we set
    $z_k(v) := z(v_k)$.
  \end{enumerate}
\end{remark}

\noindent
We briefly recall the notion of a martingale problem.

\begin{remark}[Martingale Problem]
  For some complete and separable metric space $(E,r)$, some linear
  $G: \mathcal D(G) \subseteq \mathcal C_b(E) \to \mathcal C_b(E)$ and
  $\mu \in \mathcal P(E)$, we say that an $E$-valued process $X$
  solves the $(G, \mathcal D(G), \mu)$ martingale problem if
  $X_0 \sim \mu$ and
  $$ \Big( \Phi(X_t) - \int_0^t G\Phi(X_s) ds\Big)_{t\geq 0}$$
  is a martingale for every $\Phi \in \mathcal D(X)$. We say that the
  $(G, \mathcal D(G), \mu)$ martingale problem is well-posed if there
  is a unique (in law) process $X$ which solves this martingale
  problem.
\end{remark}

\noindent
We give the martingale problem for the process $(X^N, Z^N)$ for some
$N=1,2,...$

\begin{definition}[Martingale problem for the Fleming-Viot process
  with mutation modifier and fluctuating selection\label{def:mp}]
  For $(u,v) \in ([0,\vartheta] \times I)^n$ with
  $u=(u_1,....,u_n), v=(v_1,...,v_n)$ and $1\leq k,l\leq n$, we set
  \begin{align*}
    \theta_{kl}(u) := (u_1,...,u_{l-1}, u_k, u_{l},...,u_{n-1}), \qquad \theta_{kl}(u,v) := (\theta_{kl}(u), \theta_{kl}(v)).
  \end{align*}
  For the domain of the generator of $(X^N, Z^N)$, we define the set
  of functions
  \begin{align*}
    \Pi := \{(x,z)\mapsto \Phi(x) \Psi(z):
    \Phi(x) & = \Phi^{n,\phi}(x) = \langle x^n, \phi\rangle, \Psi(z) = \Psi^{m,u}(z) = z(u_1)\cdots z(u_m),
    \\ & m,n=1,2,...,
         \phi \in \mathcal C(([0,\vartheta] \times I)^n), u=(u_1,...,u_m)\in I^m\}.
  \end{align*}
  The generator then reads
  \begin{align*}
    G_N & = G^{\text{res}} + G^{\text{mut}} + N\cdot G^{\text{sel}} + N^2 \cdot G^{\text{env}}
          \intertext{with}
          G^{\text{res}} \Phi(x)\Psi(z) & = \Psi(z)\cdot \sum_{k,l=1}^n \langle x^n, \phi \circ \theta_{kl} - \phi \rangle,\\
    G^{\text{mut}} \Phi(x)\Psi(z) & = \Psi(z)\cdot \sum_{k=1}^n \langle x^n, \pi_{k,A}\cdot
                                    (\beta_{k,B} \phi - \phi) \rangle,\\
    G^{\text{sel}} \Phi(x)\Psi(z) & = \Psi(z)\cdot \sigma \sum_{k=1}^n \langle x^{n+1}, \phi \cdot ( z_{k} - z_{n+1})\rangle,\\
    G^{\text{env}} \Phi(x)\Psi(z) & = \Phi(x)\cdot \gamma \cdot (\mathbb E_\nu[\Psi(Z)] - \Psi(z)).
  \end{align*}
  Then, for
  $S:=\mathcal P([0,\vartheta] \times I) \times \mathcal C(I)$ and
  $\mu \in \mathcal P(S)$, we call every $S$-valued process
  $(X^N, Z^N)$ such that $(X^N(0), Z^N(0))\sim \mu$ and
  $$ \Big(\Phi(X^N_t)\Psi(Z^N_t) - \int_0^t G_N\Phi(X^N_s)\Psi(Z^N_s)\Big)_{t\geq 0}$$
  is a martingale, the {\em Fleming-Viot process with mutation
    modifier and fluctuating selection}. Its martingale problem is
  called the $(G_N, \Pi, \mu)$-martingale problem.
\end{definition}

\begin{remark}[Interpretation of generator terms]
  Note that the terms $G^{\text{res}}$ and $G^{\text{sel}}$ appear
  frequently when studying Fleming-Viot systems; see e.g.  Chapter~3
  of \cite{EthierKurtz1993}. For the mutation operator, we note that
  \begin{align*}
    (\pi_{k,A}\cdot \beta_{k,B} \phi)(u,v) = u_k \cdot \int \beta(v_k, dv') \phi(u, v_1,...,v_{k-1}, v', v_{k+1},...,v_n).
  \end{align*}
  Hence, the state at the $A$-locus, $u_k$, equals the mutation rate
  at the $B$-locus.
\end{remark}

\begin{lemma}\label{l:uniN}
  For $N=1,2,...$ and $\mu \in \mathcal P(S)$, the
  $(G_N, \Pi, \mu)$-martingale problem is well-posed. This solution
  $(X^N, Z^N)$ is strongly continuous, i.e.\
  $(X^N_t, Z_t^N) \xRightarrow{t\to 0} (X^N_0, Z_0^N)$ and has the
  Feller property, i.e.\ $x\mapsto \mathbb E_x[f(X^N_t, Z_t^N)]$ is
  continuous for every $f\in\mathcal C(S)$.
\end{lemma}

\begin{proof}
  Fix $N$. First, note that for any solution $(X^N, Z^N)$ of the
  martingale problem, we see that (by setting $\Phi=1$)
  $$ \Big( \Psi(Z_t^N) -  N^2 \gamma\int_0^t (\mathbb E_\nu[\Psi(Z)] - \Psi(Z_s^N))ds \Big)_{t\geq 0}$$
  is a martingale problem. From this, we read off that $Z^N$ is a
  Markov jump process, which jumps from $z$ to $Z \sim \nu$ at rate
  $N^2\gamma$; see \cite{EthierKurtz86}, Section 4.2. Second, we can
  condition on $Z^N$ and construct $X^N$ conditional on $Z^N$. Since
  $Z^N$ is piece-wise constant, and jump points do not accumulate, we
  can solve the resulting martingale problem for $X^N$ (conditional on
  $Z^N$) uniquely between jumps of $Z^N$. Hence, we only require
  well-posedness of the martingale problem for $\gamma=0$. This,
  however, is a classical result in mathematical population genetics;
  see e.g.\ \cite{EthierKurtz1993}. In summary, by this two-step
  procedure, we obtain existence and uniqueness of the
  $(G_N, \Pi, \mu)$-martingale problem.
\end{proof}

\begin{theorem}[Convergence for fast fluctuating environment\label{T1}]
  Given that $X^N_0 \xRightarrow{N\to\infty} X_0 \sim \mu_1$ and
  $2\sigma^2/\gamma < 1$, we find that
  $X^N\xRightarrow{N\to\infty} X$, the unique solution of the
  $(G, \Pi_1, \mu_1)$ martingale problem, where
  \begin{align*}
    \Pi_1 := \{x\mapsto \Phi(x): \Phi(x) & = \Phi^{n,\phi}(x) = \langle x^n, \phi\rangle,
                                           n=1,2,...,
                                           \phi \in \mathcal C(([0,\vartheta] \times [0,1])^n)\}
  \end{align*}
  and, setting
  \begin{align}\notag
    \chi_{k,l}(v) & := \chi(v_k, v_l) := \mathbb E_\nu[Z(v_k)Z(v_l)],
                    \intertext{with}
                    \notag G & =  G^{\text{res}} + G^{\text{mut}} + \overline G^{\text{sel}}
                               \intertext{where $G^{\text{res}}$ and $G^{\text{mut}}$ are
                               as in Definition \ref{def:mp} and, for $\Phi = \Phi^{n,\phi}$,}
                               \notag   \overline G^{\text{sel}}\Phi(x)
    & = \frac{\sigma^2}{\gamma} \sum_{k,l=1 \atop k\neq l}^n\big\langle x^{n+2},
      \phi\cdot (\chi_{kl} - \chi_{n+1,n+2})\big\rangle + 2n \frac{\sigma^2}{\gamma} \sum_{k=1}^n\big\langle x^{n+2},
      \phi\cdot(\chi_{n+1,n+2} - \chi_{k,n+1} \big\rangle \notag
    \\ & \qquad \qquad \qquad \qquad \qquad \qquad \qquad \qquad \qquad +  \frac{\sigma^2}{\gamma} \sum_{k=1}^n
         \big\langle x^{n+2},
         \phi\cdot (\chi_{kk} - \chi_{n+1,n+1})\big\rangle. \label{eq:Gfsel}
  \end{align}
\end{theorem}

\begin{remark}[Techniques needed for the proof]\label{rem:178}
  The proof of Theorem~\ref{T1} is an application of Corollary~1.7.8
  in \cite{EthierKurtz86}, together with duality techniques.
  \begin{enumerate}
  \item Corollary~1.7.8 in \cite{EthierKurtz86} is dealing with
    strongly continuous contraction semigroups; see
    Remark~\ref{rem:notation}. Let us briefly recall this result. For
    some locally compact and separable $(E,r)$, let
    $L := \mathcal C_b(E)$, equipped with the topology of uniform
    convergence on compacts. For operators $G_i$ with domain
    $\mathcal D(G_i)$, $i=0,1,2$, assume the following:
    \begin{enumerate}
    \item $G_2$ generates a strongly continuous contraction semigroup
      $(S_t)_{t\geq 0}$ on $L$, such that
      $$ \lim_{\lambda\to 0+} \lambda \int_0^\infty e^{-\lambda t}S_tf dt =: Pf \text{ exists for all $f\in L$};$$
    \item
      $\mathcal D := \mathcal D(G_0) \cap \mathcal D(G_1) \cap \mathcal
      D(G_2)$ is a core for $G_2$;
    \item For $N$ sufficiently large, $G_0 + N\cdot G_1 + N^2 \cdot G_2$
      generates a strongly continuous contraction semigroup
      $(T^N(t))_{t\geq 0}$ on $L$.
    \end{enumerate}
    For
    $f \in D\subseteq \{f: \mathcal D(G_0) \cap \mathcal D(G_1):
    G_2f=0\}$, set
    $$D_f := \{h \in \mathcal D: G_2h = -G_1f\}$$
    and define for any $f\in D$ and $h\in D_f$
    \begin{align}
      \label{eq:Poi}
      \bar G f = PG_0f + PG_1 h.
    \end{align}
    Then, $\bar G$ is dissipative and if its closure generates a
    strongly continuous contraction semigroup $(T(t))_{t\geq 0}$ on
    $\bar D$, then $T^N(t)f\xrightarrow{N\to\infty} T(t)f$ for all
    $t\geq 0$, uniformly on bounded intervals.
    \\
    Let us be a bit more precise how to apply the above scenario. In
    particular, we are dealing with the special situation that
    $S = S_1\times S_2$,
    \begin{itemize}
    \item[(A1)] $G_2$ has the form, for some $\nu \in \mathcal P(S_2)$
      and $\gamma>0$,
      $$ G_2f(x,z) = \gamma \big(\mathbb E_\nu[f(x,Z)] - f(x,z)\big);$$
    \item[(A2)] $G_1$ satisfies $\mathbb E_\nu[G_1f(x,Z)]=0$ if $f$ only
      depends on $x$.
    \end{itemize}
    In this situation, $G_2$ generates a strongly continuous contraction
    semigroup $(S_t)_{t\geq 0}$ on $\mathcal C_b(E)$, which has the form
    $$ S_tf(x,z) = e^{-\gamma t} f(x,z) + (1-e^{-\gamma t}) \mathbb E_\nu[f(x,Z)].$$
    Clearly, since
    $S_tf(x,z) = e^{-\gamma t} f(x,z) + (1-e^{-\gamma t})\mathbb
    E_\nu[f(x,Z)]$,
    \begin{align*}
      \lambda \int_0^\infty e^{-\lambda t}S_tf(x,z) dt
      & =
        \frac{\lambda}{\lambda + \gamma}(f(x,z) - \mathbb E_\nu[f(x,Z)]) +
        \lambda \int_0^\infty e^{-\lambda t}\mathbb E_\nu[f(x,Z)] dt
      \\ & \xrightarrow{\lambda\to 0} \mathbb E_\nu[f(x,Z)] =:Pf(x,z).
    \end{align*}
    Then, for $G_2f=0$, we need that $(x,z)\mapsto f(x,z)$ only depends
    on $x$. In this case, we have by (A1) and (A2)
    \begin{align*}
      G_2G_1f(x,z) = \gamma \big(\mathbb E_\nu[G_1f(x,Z)] - G_1f(x,z)\big) = - \gamma G_1f(x,z),
    \end{align*}
    i.e.\ $h = \tfrac 1\gamma G_1f$ is a solution of $G_2h = -G_1f$. In
    total, we find that (abusing notation by writing $x\mapsto f(x)$ if
    $f$ only depends on $x$), \eqref{eq:Poi} transforms to
    \begin{align}
      \bar Gf(x) = \mathbb E_\nu\big[G_0f(x,Z) + \tfrac 1\gamma G_1G_1f(x,Z)\big].\label{eq:129}
    \end{align}
    If we can show that $\bar G$ generates a strongly continuous
    contraction semigroup (which is implied by well-posedness of the
    $(\bar G, D)$-martingale problem), we have convergence.
  \item It remains to show well-posedness of the $\bar G$-martingale
    problem as well as the Feller property. At least, existence of a
    solution of the martingale problem follows by general theory; see
    Chapter 4.5 of \cite{EthierKurtz86}, provided that the Markov
    processes $X^N$ with semigroups $T^N$ satisfy the compact
    containment condition. Indeed, since
    $|| \tfrac 1N h|| \xrightarrow{N\to\infty} 0$ and
    \begin{align*}
      (G_0 + N\cdot G_1 + N^2 \cdot G_2)(f + \tfrac 1N h)
      & =
        G_0 f(x) + G_1 h(x,z) + N\cdot (G_1f + G_2h)  + o(1)
      \\ & = G_0 f(x) + G_1 h + o(1),
    \end{align*}
    we find generator convergence.\\
    For uniqueness and the Feller property, we will be using a duality
    argument (see Chapter 4.4 in \cite{EthierKurtz86}). Recall that
    $X$ (i.e.\ a solution of the ($\bar G, D)$-martingale problem) is
    dual to some stochastic process $Y$ with (separable) state space
    $\Upsilon$ with respect to $H: S\times \Upsilon \to\mathbb R$
    bounded and measurable, if
    \begin{align} \notag
      \mathbb E_x[H(X_t,y)] = \mathbb E_y[H(x,Y_t)]
    \end{align}
    for all $t, x, y$. If
    $\Pi := \{H(.,y): y\in\Upsilon\} \subseteq D$ and $Y$ is a Markov
    process with generator $G_Y$, and if $H(x,.)$ is in the domain of
    $G_Y$ for all $x$, the latter equality is implied by
    \begin{align}
      \label{eq:GXGY}
      \bar GH(.,y)(x) = G_YH(x,.)(y)
    \end{align}
    since
    $$ \frac{d}{ds} \mathbb E[H(X_s, Y_{t-s})] = \mathbb E[\bar GH(., Y_{t-s})(X_s) - G_YH(X_s, .)(Y_{t-s})] = 0$$
    on a probability space where $X$ and $Y$ are independent.  If
    $\Pi$ is separating, existence of $Y$ implies uniqueness of the
    $(\bar G, D)$-martingale problem; see Proposition~4.4.7 of
    \cite{EthierKurtz86}. Moreover, if $H$ is bounded and continuous,
    we find that
    $x\mapsto \mathbb E_x[H(X_t,y)] = \mathbb E_y[H(x,Y_t)]$ is
    continuous by dominated convergence. If $\Pi$ is convergence
    determining and $Y$ is Feller, this implies that $X$ is Feller as
    well.
  \end{enumerate}
\end{remark}

\begin{proof}[Proof of Theorem~\ref{T1}]
  We use Remark~\ref{rem:178}.1 with
  $S_1 = \mathcal P([0, \vartheta] \times I)$, $S_2 = \mathcal C_L(I)$
  and $G_0 = G^{\text{res}} + G^{\text{mut}}$, $G_1 = G^{\text{sel}}$
  and $G_2 = G^{\text{env}}$. (A1) is satisfied due to the form of
  $G^{\text{env}}$ in Definition~\ref{def:mp}. If $\Phi$ only depends
  on $x$, (A2) is satisfied since $G^{\text{sel}}\Phi$ depends on $z$
  only linearly and \eqref{eq:Z0} holds. If $\Phi\Psi \in \Pi$ with
  $\Phi = \Phi^{n, \phi}, \Psi = \Psi^{m,u}$ only depends on $x$, we
  have that $\Psi=$const and $h = -\frac{1}{\gamma}G^{\text{sel}}\Phi$
  solves $G^{\text{env}} h = -G^{\text{sel}}\Phi$. Therefore,
  \eqref{eq:129} gives
  \begin{align*}
    \bar G\Phi(x) &
                    = G^{\text{res}}\Phi(x) + G^{\text{mut}}\Phi(x) + \tfrac 1 \gamma
                    \mathbb E_\nu[G^{\text{sel}}G^{\text{sel}}\Phi(x,Z)].
  \end{align*}
  In order to compute that last term, we define for $v\in I^n$
  $$ \chi_{k,l}(v) := \chi(v_k, v_l) := \mathbb E_\nu[Z(v_k)Z(v_l)]$$ and obtain, for $\phi$
  depending only on the first $n$ coordinates at both loci
  \begin{align*}
    \overline G^{\text{sel}}\Phi(x) & := \frac{1}{\gamma} \mathbb E_\nu [G^{\text{sel}} G^{\text{sel}} \Phi(x,Z)]
    \\ &
         = \frac{\sigma}{\gamma}\sum_{l=1}^n \mathbb E_\nu \big[G^{\text{sel}} \big\langle x^{n+1},
         \phi\cdot (Z_{l} - Z_{n+1})\big\rangle \big]
    \\ & = \frac{\sigma^2}{\gamma} \sum_{l=1}^n \sum_{k=1}^{n+1} \mathbb E_\nu \Big[\big\langle x^{n+2},
         \phi\cdot (Z_{l} - Z_{n+1})\cdot (Z_{k} - Z_{n+2})\big\rangle\Big]
    \\ & = \frac{\sigma^2}{\gamma} \sum_{k,l=1 \atop k\neq l}^n\big\langle x^{n+2},
         \phi\cdot (\chi_{kl} - 2\chi_{k,n+1} + \chi_{n+1,n+2})\big\rangle
    \\ & \qquad + \frac{\sigma^2}{\gamma} \sum_{l=1}^n\big\langle x^{n+2},
         \phi\cdot (\chi_{l,l} - 2 \chi_{l,n+1} + 2 \chi_{n+1,n+2} - \chi_{n+1,n+1} )\big\rangle
    \\ & = \frac{\sigma^2}{\gamma} \sum_{k,l=1 \atop k\neq l}^n\big\langle x^{n+2},
         \phi\cdot (\chi_{kl} - \chi_{n+1,n+2})\big\rangle + 2n \frac{\sigma^2}{\gamma} \sum_{k=1}^n\big\langle x^{n+2},
         \phi\cdot(\chi_{n+1,n+2} - \chi_{k,n+1} \big\rangle
    \\ & \qquad \qquad \qquad \qquad \qquad \qquad \qquad \qquad \qquad +  \frac{\sigma^2}{\gamma} \sum_{k=1}^n
         \big\langle x^{n+2},
         \phi\cdot (\chi_{kk} - \chi_{n+1,n+1})\big\rangle.
  \end{align*}
  (We have used the symmetry relationship
  $\langle x^{n+2}, \phi \cdot Z_{n+1}\rangle = \langle x^{n+2}, \phi
  \cdot Z_{n+2}\rangle$.) This already establishes the form of the
  generator appearing in Theorem~\ref{T1} and existence of the
  $(G, \Pi_1)$-martingale problem follows as in
  Remark~\ref{rem:178}.2.

  For uniqueness, we use duality. The dual process will be similar to
  the one of the tree-valued Fleming-Viot process with mutation and
  selection given in \cite{depperschmidt2012}. The goal is to use
  \eqref{eq:GXGY}, and therefore, we have to rewrite the generator
  terms. We define for $u=(u_1,u_2,...)$
  \begin{align*}
    \bar{\sigma}_l(u) & = (u_{i-1_{\{i > l \}}}) = (u_1,...,u_{l},u_{l},u_{l+1},...),
    \\
    {\sigma}_l(u) & = (u_{i+1_{\{i \geq l \}}}) = (u_1,...,u_{l-1},u_{l+1},u_{l+2},...).
  \end{align*}
  We note that, for $\phi$ depending only on the first $n$
  coordinates, and $1\leq k\neq l\leq n$
  \begin{align*}
    \langle x^n, \phi \circ \theta_{kl}\rangle
    & =
      \langle x^{n-1}, \phi \circ \theta_{kl} \circ \bar\sigma_l \rangle,
    \\
    \mathbb E_\nu [\langle x^{n+1}, \phi \cdot Z_{n+1}\rangle]
    & =
      \mathbb E_\nu [ \langle x^{n+1}, (\phi \circ \sigma_k)\cdot Z_k\rangle],
    \\
    \langle x^{n+2}, \phi \cdot \chi_{n+1, n+2}\rangle
    & = \langle x^{n+2}, (\phi \circ \sigma_k)\cdot \chi_{k,n+2}\rangle =
      \langle x^{n+2}, (\phi \circ \sigma_k \circ \sigma_l)\cdot \chi_{k,l}\rangle,
  \end{align*}
  holds, since integrating with respect to the product measure $x^n$
  does not depend on the order of coordinates.

  Therefore, we can write for $\Phi = \Phi^{n,\phi}$
  \begin{equation}
    \label{eq:213}
    \begin{aligned}
      G^{\text{res}}\langle x^n, \phi\rangle & = \sum_{k,l=1 \atop
        k\neq l}^n \langle x^{n-1}, \phi \circ\theta_{k,l} \circ
      \bar\sigma_l\rangle - \langle x^n , \phi\rangle, \\
      G^{\text{mut}} \langle x^n, \phi\rangle & = \vartheta \cdot
      \sum_{k=1}^n \Big\langle x^n, \frac{\pi_{k,A}}{\vartheta} \cdot
      \beta_{k,B}\phi + \Big(1 -
      \frac{\pi_{k,A}}{\vartheta}\Big)\cdot\phi\Big\rangle - \langle
      x^n, \phi\rangle, \\ \overline G^{\text{sel}}\langle x^n,\phi
      \rangle &= \frac{\sigma^2}{\gamma} \sum_{k,l=1 \atop k\neq l}^n
      (\langle x^{n+2}, \phi \cdot \chi_{kl} + (\phi \circ \sigma_k
      \circ \sigma_l)
      \cdot(1-\chi_{kl})\rangle - \langle x^n, \phi  \rangle)\\
      & + 2n \frac{\sigma^2}{\gamma}
      \sum_{k=1}^n ( \langle x^{n+2},  (\phi \circ \sigma_k) \cdot \chi_{k,n+2} + \phi \cdot (1-\chi_{k,n+2})  \rangle - \langle x^n, \phi  \rangle) \\
      & + \frac{\sigma^2}{\gamma} \sum_{k=1}^n \big(\langle x^{n+2},
      (\phi \cdot \chi_{k,k} + (\phi \circ \sigma_k) \cdot
      (1-\chi_{k,k}) \rangle - \langle x^n, \phi \rangle).
    \end{aligned}
  \end{equation}
  With this reformulation, we can construct a function-valued dual
  process as follows. Taking the state space
  \begin{align*}
    \Upsilon = \bigcup_{n=0}^\infty
    \Upsilon_n, \qquad \Upsilon_n = \mathcal C(([0,\vartheta] \times [0,1])^n),
  \end{align*}
  we consider a pure jump process $\Xi = (\xi_t)_{t\geq 0}$ with
  transitions from $\xi \in \Upsilon_n$ to
  \begin{align*}
    \xi \circ \theta_{k,l} \circ \bar \sigma_l & \in \Upsilon_{n-1}
                                                 \text{ at rate 1 for each unordered pair $1\leq k\neq l\leq n$},
    \\\frac{\pi_{k,A}}{\vartheta}
    \cdot \beta_{k,B}\cdot \xi +
    \Big(1 - \frac{\pi_{k,A}}{\vartheta}\Big)\cdot\xi & \in \Upsilon_n
                                                         \text{ at rate $\vartheta$ for each  $1\leq k\leq n$},
    \\
    \xi \cdot \chi_{kl} + (\xi \circ \sigma_k \circ \sigma_l)
    \cdot(1-\chi_{kl}) & \in \Upsilon_{n+2} \text{ at rate $ \frac{\sigma^2}{\gamma} $ for each unordered pair $1\leq k\neq l\leq n$},
    \\ (\xi \circ \sigma_k) \cdot \chi_{k,n+2} + \xi \cdot (1-\chi_{k,n+2}) & \in \Upsilon_{n+2} \text{  at rate $ 2n \frac{\sigma^2}{\gamma} $ for each $1\leq k\leq n$},
    \\ \xi \cdot \chi_{k,k} + (\xi \circ \sigma_{k}) \cdot (1-\chi_{k,k}) & \in \Upsilon_{n+1} \text{  at rate $ \frac{\sigma^2}{\gamma} $ for each $1\leq k\leq n$}.
  \end{align*}
  Then, for $H: S \times \Upsilon$, given by
  $H(x,\xi) = \langle x^n, \xi\rangle$ for $\xi\in\Upsilon_n$, we have
  established \eqref{eq:GXGY}, i.e.\ the generator of $\Xi$ for
  $\xi\in\Upsilon_n$ is
  $(G^{\text{res}} + G^{\text{mut}} + \overline G^{\text{sel}})
  \langle x^n, \xi\rangle$ with $G^{\text{res}}, G^{\text{mut}}$ and
  $\overline G^{\text{sel}}$ as the right hand sides in
  \eqref{eq:213}. In other words, $\Xi$ and $X$, a solution of the
  $G$-martingale problem are dual, provided that existence for $\Xi$
  can be guaranteed. Here, we have to take into account that the
  number of dependent variables, $n$, can explode. This number
  decreases at rate $n(n-1)$ and increases by two at rate
  $n(n+1)\sigma^2/\gamma$ and by one at rate $\sigma^2/\gamma
  n$. Therefore, explosion cannot occur for $2\sigma^2/\gamma < 1$ and
  from Proposition~4.4.7 of \cite{EthierKurtz86}, uniqueness for the
  $G$-martingale problem follows in this case. Since
  $\{H(.,\xi): \xi\in\Upsilon\}$ is separating and convergence
  determining (see e.g.\ Example~5 in
  \citealp{DepperschmidtGrevenPfaffelhuber2019}), we have shown that
  $\bar G$ generates a strongly continuous contraction semigroup and
  the proof of Theorem~\ref{T1} is complete; see
  Remark~\ref{rem:178}.2.
\end{proof}

\section{Specialization to a finite dimensional system}
\label{sec_3}
We will now specialize Theorem~\ref{T1} to a finite-dimensional
system. Precisely, since we have two loci, the minimal number of
dimensions is $2\times 2$.  So, only four types will be present, which
will be denoted ${\ell0}, {\ell1}, {h0}, {h1}$. For
$0\leq \vartheta_\ell \leq \vartheta_h \leq\vartheta$, their
frequencies are given through
$x \in \mathcal P([0,\vartheta] \times I)$ by
\begin{align} \notag
  x_{ai} := \Phi_{ai}(x) :=
  x(\{\vartheta_a\} \times \{i\}) = \langle x, 1_{\{\vartheta_a\} \times \{i\}} \rangle,
  \qquad (a,i) \in \{\ell, h\}\times \{0,1\}.
\end{align}
For mutation, we consider the case that each mutation event (either at
rate $\vartheta_\ell$ or $\vartheta_h$) results in type~0 at the
$B$-locus with probability $r \in [0,1]$. For selection, let
$z: \{0,1\} \to \{-\tfrac 12, \tfrac 12\}$ be given by
$z(0) = \tfrac 12, z(1) = -\tfrac 12$ and
\begin{align*}
  \nu = \tfrac 12 (\delta_{z} + \delta_{-z}).
\end{align*}
Consider the solution $X^N$ of the martingale problem from
Definition~\ref{def:mp} in this case, which exists uniquely by
Lemma~\ref{l:uniN}. Letting
$X^N_{ai}, (a,i) \in \{\ell,h\} \times\{0,1\}$ be as above, using the
martingale representation theorem (see e.g.\ Theorem~16.12. of
\citealp{Kallenberg2002}), it is straight-forward to see that
$X^N = (X^N_{\ell 0}, X^N_{\ell 1}, X^N_{h0}, X^N_{h1})$ is a weak
solution of the system of SDEs
\begin{equation}
  \label{eq:sSDEs}
  \begin{aligned}
    dX^N_{\ell0} & = \sigma N Z^N X^N_{\ell0}X^N_1 dt + \theta_\ell
    (rX^N_{\ell1}-(1-r)X^N_{\ell0}) dt \\ & \qquad \qquad \qquad
    \qquad \qquad + \sqrt{X^N_{\ell0}X^N_{\ell1}} dW_{1} +
    \sqrt{X^N_{\ell0}X^N_{h0}} dW_{2} + \sqrt{X^N_{\ell0}X^N_{h1}}
    dW_{3},
    \\
    dX^N_{\ell1} & = - \sigma N Z^N X^N_{\ell1}X^N_0 dt + \theta_\ell
    ((1-r)X^N_{\ell0}-rX^N_{\ell1}) dt \\ & \qquad \qquad \qquad
    \qquad \qquad - \sqrt{X^N_{\ell1}X^N_{\ell0}} dW_{1} +
    \sqrt{X^N_{\ell1}X^N_{h0}} dW_{4} + \sqrt{X^N_{\ell1}X^N_{h1}}
    dW_{5},\\
    dX^N_{h0} & = \sigma N Z^N X^N_{h0}X^N_1 dt + \theta_h
    (rX^N_{h1}-(1-r)X^N_{h0}) dt \\ & \qquad \qquad \qquad \qquad
    \qquad - \sqrt{X^N_{h0}X^N_{\ell0}} dW_{2} -
    \sqrt{X^N_{h0}X^N_{\ell1}} dW_{4} + \sqrt{X^N_{h0}X^N_{h1}}
    dW_{6},
    \\
    dX^N_{h1} & = - \sigma N Z^N X^N_{h1}X^N_{0} dt + \theta_h
    ((1-r)X^N_{h0}-rX^N_{h1}) dt \\ & \qquad \qquad \qquad \qquad
    \qquad - \sqrt{X^N_{h1}X^N_{\ell0}} dW_{3} -
    \sqrt{X^N_{h1}X^N_{\ell1}} dW_{5} - \sqrt{X^N_{h1}X^N_{h0}}
    dW_{6},
  \end{aligned}
\end{equation}
with $X^N_i = X^N_{hi} + X^N_{\ell i}$, $i=0,1$, independent Brownian
motions $W_1,...,W_6$, and $Z^N$ (the fitness difference between
types~0 and 1) changes from $-1$ to $+1$ and back at rate
$N^2 \tfrac \gamma 2$.

\begin{theorem}[Convergence for fast fluctuating environment\label{T2}]
  \sloppy For weak solutions $(X^N)_{N=1,2,...}$ of \eqref{eq:sSDEs},
  assume that $X^N(0) \xRightarrow{n\to\infty}X_0$ and
  $2\sigma^2/\gamma < 1$. Then,
  $(X^N_{\ell0}, X^N_{\ell1}, X^N_{h0}, X^N_{h1})
  \xRightarrow{N\to\infty} X = (X_{\ell0}, X_{\ell1}, X_{h0},
  X_{h1})$, the unique weak solution of
  \begin{equation}
    \label{eq:limit}
    \begin{aligned}
      dX_{\ell0} & = \tfrac{\sigma^2}{\gamma}X_{\ell0}X_1(X_1-X_0)dt +
      \theta_\ell (rX_{\ell1} - (1-r)X_{\ell0}) dt \\ & \qquad \qquad
      \qquad + \sqrt{X_{\ell0}X_{\ell1}} dW_{1} +
      \sqrt{X_{\ell0}X_{h0}} dW_{2} + \sqrt{X_{\ell0}X_{h1}} dW_{3} +
      \sigma\sqrt{2/\gamma} X_{\ell0}X_1 dW,
      \\
      dX_{\ell1} & = \tfrac{\sigma^2}{\gamma}X_{\ell1}X_0(X_0-X_1) dt
      + \theta_\ell ((1-r)X_{\ell0} - rX_{\ell1}) dt \\ & \qquad \qquad
      \qquad - \sqrt{X_{\ell1}X_{\ell0}} dW_{1} +
      \sqrt{X_{\ell1}X_{h0}} dW_{4} + \sqrt{X_{\ell1}X_{h1}} dW_{5}-
      \sigma\sqrt{2/\gamma} X_{\ell1}X_0 dW
      \\
      dX_{h0} & = \tfrac{\sigma^2}{\gamma}X_{h0}X_1(X_1-X_0)dt +
      \theta_h (rX_{h1} - (1-r)X_{h0}) dt \\ & \qquad \qquad \qquad -
      \sqrt{X_{h0}X_{\ell0}} dW_{2} - \sqrt{X_{h0}X_{\ell1}} dW_{4} +
      \sqrt{X_{h0}X_{h1}} dW_{6} + \sigma\sqrt{2/\gamma} X_{h0}X_1 dW,
      \\
      dX_{h1} & = \tfrac{\sigma^2}{\gamma}X_{h1}X_0(X_0-X_1) dt +
      \theta_h ((1-r)X_{h0}-rX_{h1}) dt \\ & \qquad \qquad \qquad -
      \sqrt{X_{h1}X_{\ell0}} dW_{3} - \sqrt{X_{h1}X_{\ell1}} dW_{5} -
      \sqrt{X_{h1}X_{h0}} dW_{6}- \sigma\sqrt{2/\gamma} X_{h1}X_0 dW,
    \end{aligned}
  \end{equation}
  with independent Brownian motions $W,W_1,...,W_6$ with initial
  condition $X_0$.
\end{theorem}

\begin{remark}[Evolution of $X_h$ and $X_0$]
  Writing $X_h = X_{h0} + X_{h1}$ and and $X_\ell = 1-X_h$, we also
  have
  \begin{align}\label{eq:limitCor}
    dX_h & = \frac{\sigma^2}{\gamma}(X_{h0}X_{\ell1} - X_{h1}X_{\ell0})(X_1 - X_0) dt + \sqrt{X_hX_\ell} dW'
           + \sigma\sqrt{2/\gamma}
           (X_{h0}X_{\ell1} - X_{h1}X_{\ell0}) dW,
  \end{align}
  with independent Brownian motions $W, W'$.  In the same way we can
  set $X_0 = X_{h0} + X_{\ell0}$ and $X_1 = 1-X_0$, and get
  \begin{align*}
    dX_0 & = \frac{\sigma^2}{\gamma} X_0 X_1 (X_1 - X_0) dt
           + \vartheta_\ell (r - X_0) + (\vartheta_h - \vartheta_\ell) (r X_{h1} - (1-r)X_{h0}) dt
    \\ & \qquad \qquad \qquad \qquad \qquad \qquad \qquad \qquad \qquad \qquad + \sqrt{X_0 X_1} dW''
         + \sigma\sqrt{2/\gamma}
         X_0X_1 dW
  \end{align*}
  with independent Brownian motions $W, W''$.
\end{remark}

\begin{remark}[Comparison with \cite{Gillespie1981}]
  Gillespie has considered a similar diffusion for a mutation modifier
  locus in diploids~\cite{Gillespie1981}.  While the mutation rates
  differ in Gillespie's model compared to the as we do not consider
  heterozygotes in our haploid model, the remaining diffusion terms of
  a symmetric semi-dominant model from Gillespie are similar to our
  setting.

 To see this consider equation (5) in \cite{Gillespie1981}.
 The variable $p_1 = 1 - q_1$ corresponds to our $X_0$, and $p_2 = 1 - q_2$ to $X_h$.
 In the symmetric semi-dominant model Gillespie set $A=0$ and $B=2$.
 Thus, ignoring all terms with mutation rates, %
 we get
 \begin{align*}
 dp_1 & = p_1q_1\left( A + B(\frac12 - p_1) \right)  + p_1q_1 dW\\
  & = X_0X_1 (X_1 - X_0) dt + X_0X_1 dW, \text{ and}\\
  dp_2 & = D\left( A + B(\frac12 - p_1) \right) dt + D dW
     = D (X_1 - X_0) dt + D dW
 \end{align*}
 for linkage disequilibrium $D := (X_{h0}X_{\ell1} - X_{h1}X_{\ell0})$.
 Furthermore, we can use It\^o's lemma to get
 \begin{align*}
 d(X_{h0}X_{\ell1}) & = X_{\ell1} X_{h0} (X_1 - X_0)^2 dt  - X_{h0} X_1 X_{\ell1} X_0 dt + X_{h0}X_{\ell1}(X_1 - X_0)dW
 \end{align*}
 and
 \begin{align*}
 dD & = d(X_{h0}X_{\ell1} - X_{h1}X_{\ell0}) \\
 & = (X_{h0}X_{\ell1} - X_{h1}X_{\ell0})(X_1 - X_0)^2 dt  -  (X_{h0}X_{\ell1} - X_{h1}X_{\ell0}) X_1 X_0 dt \\
 & \qquad + (X_{h0}X_{\ell1} - X_{h1}X_{\ell0})(X_1 - X_0)dW \\
 & = D(q_1 - p_1)^2 dt  -  D p_1q_1 dt + D(p_1 - q_1)dW
 \end{align*}
  The special case presented here is thus a haploid version of the symmetric semi-dominant model in Gillespie's work.
\end{remark}

\begin{proof}[Proof of Theorem~\ref{T2}]
  Since $X^N$ weakly solves \eqref{eq:sSDEs} if and only if it solves
  the martingale problem from Definition~\ref{def:mp}, we need to show
  that a solution of the limiting martingale problem from
  Theorem~\ref{T1} solves \eqref{eq:limit}. By the martingale
  representation Theorem (see e.g.\ Theorem~16.12. of
  \citealp{Kallenberg2002}), it is enough to show that (with
  $X = (X_{\ell0}, X_{\ell1}, X_{h0}, X_{h1})$ a solution of the
  limiting martingale problem) $X$ is a semimartingale with
  $X=X_0 + M+A$, where $A = (A_{\ell0}, A_{\ell1}, A_{h0}, A_{h1})$ is
  a process of finite variation with
  \begin{equation}
    \label{eq:toshow1}
    \begin{aligned}
      A_{a 0}(t)
      & = \int_0^t
      \theta_a (rX_{a,1}(s) - (1-r)X_{a0}(s)) + \frac{\sigma^2}{\gamma}
      X_{a0}(s)X_{1}(s)(X_{1}(s) - X_0(s))ds,
      \\
      A_{a 1}(t)
      & = \int_0^t
      \theta_a ((1-r)X_{a0}(s) - r X_{a1}(s)) + \frac{\sigma^2}{\gamma}
      X_{a1}(s)X_{0}(s)(X_{0}(s) - X_1(s))ds,
    \end{aligned}
  \end{equation}
  and  $M = (M_{\ell0}, M_{\ell1}, M_{h0}, M_{h1})$ is a
  martingale with
  covariation
  \begin{align}
    \label{eq:toshow2} [M_{ai}, M_{bj}](t)
    & = \int_0^t \Big((\delta_{ai, bj} -X_{ai}(s))X_{bj}(s) +
      (-1)^{i+j}\frac{2\sigma^2}{\gamma} X_{ai}(s) X_{bj}(s)X_{1-i}(s)X_{1-j}(s)\Big) ds.
  \end{align}
  As a general fact (see e.g.\ Corollary 4.6 in
  \citealp{depperschmidt2012}),
  \begin{align}\label{eq:genFact1}
    A_{ai}(t) & = \int_0^t G\Phi_{ai}(X(s))ds,\\ \label{eq:genFact2}
    [M_{ai}, M_{bj}](t) & = \int_0^t G\Phi_{ai}\Phi_{bj}(X(s)) - \Phi_{ai}(X(s)) G\Phi_{bj}(X(s)) -
                          \Phi_{bj}(X(s))G\Phi_{ai}(X(s))ds.
  \end{align}
  While the first term in \eqref{eq:toshow1} is due to
  $G^{\text{mut}}$, the first term in \eqref{eq:toshow2} is due to
  $G^{\text{res}}$. For the remaining terms, we need to evaluate the
  operator $\bar G^{\text{sel}}$.  First, for $v \in \{0,1\}^n$ and
  $Z\sim \nu$,
  \begin{align*}
    \chi_{kl}(v) = \mathbb E_\nu[Z(v_k) Z(v_\ell)]
    = \tfrac 14 \big(1_{v_k = v_l} - 1_{v_k \neq v_l}\big)
    = \tfrac 12 1_{v_k = v_l} - \tfrac 14.
  \end{align*}
  Plugging this into \eqref{eq:Gfsel}, we obtain
  \begin{align*}
    \overline G^{\text{sel}} \Phi_{ai}(x)
    & = \frac{\sigma^2}{\gamma} \langle x^3,
      1_{\{\vartheta_a\} \times \{i\}}(u_1,v_1) (1_{v_2 = v_3} - 1_{v_1 = v_2})\rangle,\\
    & =
      \frac{\sigma^2}{\gamma} \big(x_{ai}(1 - 2x_0 x_1) - x_{ai}x_i\big)
      = \frac{\sigma^2}{\gamma} x_{ai}(x_{1-i} - 2x_ix_{1-i})) = \frac{\sigma^2}{\gamma}x_{ai}x_{1-i}(x_{1-i} - x_i) ,
      \intertext{which shows \eqref{eq:toshow1} due to \eqref{eq:genFact1} and}
      \overline G^{\text{sel}} \Phi_{ai} \Phi_{bj}(x)
    & -
      \Phi_{ai}(x) \overline G^{\text{sel}} \Phi_{bj}(x)
      - \Phi_{bj}(x) \overline G^{\text{sel}} \Phi_{ai}(x)
    \\ & = \frac{\sigma^2}{\gamma} \Big(\langle x^2,
         1_{\{\vartheta_a\} \times \{i\}}(u_1,v_1)
         1_{\{\vartheta_b\} \times \{j\}}(u_2,v_2) 1_{v_1 = v_2}\rangle - x_{ai}x_{bj} (1 - 2x_0x_1)
    \\ & \qquad \qquad \qquad \qquad \qquad \qquad \qquad + x_{ai}{x_{bj}}(x_{1-i}(x_{1-i} - x_i) + x_{1-j}(x_{1-j} - x_j))\Big)
    \\ & = \frac{\sigma^2}{\gamma} x_{ai}x_{bj} (\delta_{ij}-1 + x_{1-i}^2 + x_{1-j}^2)
         \intertext{Now, for $i = j$, this gives}
    & = \frac{2\sigma^2}{\gamma} x_{ai}x_{bi}x_{1-i}^2,
      \intertext{whereas for $i \neq j$, we have}
    & = \frac{\sigma^2}{\gamma} x_{ai}x_{bj}(-1 + x_0^2 + x_1^2) = - \frac{2\sigma^2}{\gamma} x_{ai}x_{bj}x_0x_1,
      \intertext{which gives in total}
       & = (-1)^{i+j}\frac{2\sigma^2}{\gamma} x_{ai}x_{bj}x_{1-i}x_{1-j},
  \end{align*}
  which finally gives \eqref{eq:toshow2} due to \eqref{eq:genFact2}
  and the proof is complete.
\end{proof}

~

\noindent
Recall $X_h = X_{h0} + X_{h1}$ and $X_\ell = 1-X_h$. We now give a
result on the fixation probability of $X_h$.

\begin{theorem}[Fixation probability\label{T3}]
  Let $X$ be the solution of \eqref{eq:limit} with initial condition
  \begin{align*}
    X_h(0) & = x, \qquad & X_{h0}(0) &= px, \qquad & X_{\ell0}(0) &= q(1-x),\\
    X_{\ell}(0) &= (1-x),  \qquad & X_{h1}(0) &= (1-p)x, \qquad & X_{\ell1}(0) &= (1-q)(1-x).
  \end{align*}
  Let $r\in[0,1]$ be the probability that a mutation event results in type 0.
  Then,
  \begin{align}
    \label{eq:T3}&\frac{\gamma}{\sigma^2}(\mathbb P  (X_h(\infty)=1) - x)
    \\  \notag & \quad \xrightarrow{\sigma^2/\gamma \to 0}
                 x(1-x) \Big[\frac{(2r-1)(q-r)(1 + 2\vartheta_\ell)}{(1+\vartheta_\ell)(3 + 2\vartheta_\ell)}
                 - \frac{(2r-1)(p-r)(1+2\vartheta_h)}{(1+\vartheta_h)(3+2\vartheta_h)}
    \\ \notag & \quad + 2\Big(
                (1-x) \frac{(q-r)^2}{3 + 2\vartheta_\ell}
                - x\frac{(p-r)^2}{3 + 2\vartheta_h}+ (2x-1)\frac{(p-r)(q-r)}{3 + \vartheta_\ell + \vartheta_h}+ r(1-r)
                \Big(\frac{1}{3 + 2\vartheta_\ell} - \frac{1}{3 + 2\vartheta_h}\Big)
                \Big)\Big].
  \end{align}
\end{theorem}

\noindent
Actually, a straight-forwards (but tedious) calculation leads to a
different form of the last formula.

\begin{corollary}[Different form of the fixation probability\label{cor1}]
  For the same situation as in Theorem~\ref{T3}, \eqref{eq:T3} can
  also be written as
  \begin{align}\label{eq:cor1}
    \frac{\gamma}{\sigma^2}&(\mathbb P  (X_h(\infty)=1) - x) \xrightarrow{\sigma^2/\gamma \to 0}
    \\  &x(1-x) \notag\\\notag
                           &\cdot\Bigg[ (p-q)\cdot
                             \left(\frac{(1-2r)(1+2 \vartheta_l)}{(3+2\vartheta_l)(1+\vartheta_l)}
                             + \frac{2(1-x)(r-q)}{(3+2\vartheta_l)} + \frac{2x(r-p)}{(3+2 \vartheta_h)}\right)\\
                           \notag&\hspace{0.4cm}- (1-2r)(r-p)(\vartheta_h - \vartheta_l)\\
                           \notag&\hspace{0.7cm}\cdot\left(
                             - \frac{2(7+ 2 \vartheta_l +2 \vartheta_h)}{(2+\vartheta_h)(3+2\vartheta_h )(2+\vartheta_l)(3+2\vartheta_l)}
                             + \frac{(2 - \vartheta_h\vartheta_l)}{(2+\vartheta_l)(1+\vartheta_l)(2+\vartheta_h)(1+\vartheta_h)}\right)\\
                           \notag&\hspace{0.4cm}+\frac{2(r - q)(r-p)(\vartheta_h-\vartheta_l)}{(3 + \vartheta_h + \vartheta_l)} \cdot \left( \frac{(1-x)}{(3 + 2\vartheta_l)} + \frac{x}{(3+2\vartheta_h)}\right)
                             +\frac{4r(1-r) (\vartheta_h - \vartheta_l)}{(3+2\vartheta_l)(3+2\vartheta_h )}
                             \Bigg].
  \end{align}
\end{corollary}

\begin{remark}[Checking the fixation probability]
  Some symmetries in \eqref{eq:T3} (or equivalently in
  \eqref{eq:cor1}) can directly be seen:
  \begin{itemize}
  \item The right hand side changes sign if we exchange
    $\vartheta_h \leftrightarrow \vartheta_\ell$, $p\leftrightarrow q$ and
    $x\leftrightarrow 1-x$, since the roles of $X_h$ and $X_\ell$ are
    simply exchanged.
  \item If $p=q=r=0$ or $p=q=r=1$, the right hand side is 0.
  \item If $\vartheta_h = \vartheta_\ell=0$, the result does not depend on
    $r$ since there are no mutations.
  \item If $\vartheta_h = \vartheta_\ell$ and $p=q$, the right hand side
    is~0 since $X_h$ and $X_\ell$ are the same (in distribution).
  \end{itemize}
  Another interesting case is $p=q=r$, which means that both $X_h$ and
  $X_\ell$ are in their mutational balance already at time~0. In this
  case, we find that
  $$ \mathbb P  (X_h(\infty)=1) \approx x
  + 4 x(1-x) \frac{\sigma^2}{\gamma} \frac{r(1-r) (\vartheta_h -
    \vartheta_l)}{(3+2\vartheta_l)(3+2\vartheta_h )}$$ for small
  $\sigma^2 / \gamma$. This means that the fixation probability of
  $X_h$ is greater than under neutrality (i.e.\ for $\sigma^2=0$) iff
  $\vartheta_h > \vartheta_\ell$.
\end{remark}

\begin{remark}[Computing moments under neutrality]\label{rem:neutral}
  In the proof of Theorem~\ref{T3}, we will have to compute moments of
  $X$ under neutral evolution, i.e.\ $\sigma^2/\gamma=0$ in
  \eqref{eq:limit}. Since the evolution of $X$ is only driven by
  mutation and resampling then, such moments can be computed using the
  coalescent \citep{Durrett2008}, which is dual to the solution of
  \eqref{eq:limit}. Assume we aim to compute an $n$-th moment of
  $X(t)$ i.e.\ $\mathbb E[X_{a_1 i_1}(t) \cdots X_{a_n i_n}(t)]$ for
  some $a_1,...,a_n \in \{\ell,h\}$ and $i_1,...,i_n \in
  \{0,1\}$. Then, the coalescent starts with $n$ lineages, any
  (unordered) pair of lineages coalesces independently at rate~1, and
  the resulting lineages, stopped after having evolved for time $t$,
  are assigned some type, randomly chosen from $X(0)$. Mutations are
  modeled on top of this tree structure, and we have to deal with all
  cases such that lineage $k$ is assigned type $a_ki_k,
  k=1,...,n$. Since there is no mutation transforming $\ell$ to $h$
  and back, lineages assigned with $\ell$ must not coalesce with
  lineages with $h$, and ancestors of $\ell$ ($h$) must be of type
  $\ell$ ($h$). On all such events, mutation from $0$ to $1$ and back
  (at rates $\vartheta_h$ and $\vartheta_\ell$, depending on the type
  at the first locus) determines types at the second locus. These
  arguments will be used below starting in \eqref{eq:311}.
\end{remark}

\begin{proof}[Proof of Theorem~\ref{T3}]
  We will use the equality (recall \eqref{eq:limitCor})
  \begin{align}\label{eq:T3int}
    \mathbb P_x(X_h(\infty)=1)
    & =
      \mathbb E_x[X_h(\infty)] = x + \int_0^\infty \mathbb E[GX_h(t)] dt
    \\ & \notag = x + \frac{\sigma^2}{\gamma}
         \int_0^\infty \mathbb E[(X_{h0}(t)X_{\ell1}(t) - X_{h1}(t)X_{\ell0}(t))(X_1(t) - X_0(t))] dt,
         \intertext{together with}
         \notag (X_{h0}X_{\ell1} - X_{h1}X_{\ell0}) & (X_1 - X_0)
                                                      =
                                                      (X_{h0}X_{\ell1} + X_{h0}X_{\ell0} - X_{h1}X_{\ell0} - X_{h0}X_{\ell0})(1 - 2X_0)
    \\ \notag
    & = (X_\ell X_{h0} - X_hX_{\ell0}) + 2((X_hX_{h0}X_{\ell0} - X_\ell X_{h0}X_{\ell0}) + (X_hX_{\ell0}^2 - X_\ell X_{h0}^2)).
  \end{align}
  Since we are studying the case of low $\sigma^2/\gamma$, and the
  integral in \eqref{eq:T3int} is continuous in $\sigma^2/\gamma$, we
  only need to evaluate the integral at $\sigma^2/\gamma=0$. From
  \eqref{eq:limit}, we see that we need to study neutral evolution
  with the same mutation mechanism. We will write $\mathbb P(.)$ for
  the corresponding probability measure and $\mathbb E[.]$ for the
  expectation under neutral evolution. Following
  Remark~\ref{rem:neutral}, we start with
  \begin{align}
    \notag \mathbb E[X_{h}(t)] & = X_{h}(0), \qquad \mathbb E[X_{\ell}(t)] = X_{\ell}(0)\\
    \label{eq:311}
    \mathbb E[X_{h0}(t)] & = e^{-\vartheta_h t} X_{h0}(0) + (1-e^{-\vartheta_h t}) rX_h(0) = x(r + e^{-\vartheta_h t}(p-r)),
    \\
    \notag \mathbb E[X_{\ell 0}(t)]
                               & = (1-x)(r + e^{-\vartheta_\ell t}(q-r)),
  \end{align}
  since either no mutation at the $B$-locus happened by time $t$ and
  the ancestor at time $0$ had type 0, or a mutation occurred which
  resulted in a type~0 at the $B$-locus. Then, for
  $\mathbb E[X_{\ell}(t)X_{h0}(t)]$, note that coalescence of the two
  corresponding lines must not have occurred by time $t$ since
  mutation cannot transform $\ell$ to $h$ or back. The same argument
  applies to $\mathbb E[X_{h}(t)X_{\ell0}(t)]$, hence,
  \begin{align}
    \notag \int_0^\infty & \mathbb E[X_\ell(t) X_{h0}(t) - X_h(t) X_{\ell0}(t)] dt
    \\ \notag & = \int_0^\infty e^{-t} ((1-x)x (r + e^{-\vartheta_h t}(p-r)) - x(1-x) (r + e^{-\vartheta_\ell t}(q - r))dt
    \\ \label{eq:1} & = x(1-x) \Big(\frac{p-r}{1+\vartheta_h} - \frac{q-r}{1+\vartheta_\ell}\Big).
                      \intertext{For $\mathbb E[X_h X_{h0} X_{\ell0} - X_\ell X_{h0} X_{\ell0}]$, coalescence may
                      occur between the two $h$-lines in the first and the two $\ell$-lines in the second term. However,
                      on the event that such a coalescence occurs, $\mathbb E[X_h X_{h0} X_{\ell0},
                      \text{coal}] = \mathbb E[X_{h0} X_{\ell0}, \text{coal}] = \mathbb E[X_\ell X_{h0} X_{\ell0},
                      \text{coal}]$, i.e.\ this case cancels. Hence,}
                      \notag \int_0^\infty & \mathbb E[X_h(t) X_{h0}(t) X_{\ell0}(t) - X_\ell(t) X_{h0}(t) X_{\ell0}(t)] dt
    \\ \notag& = \int_0^\infty e^{-3t} x(1-x)(2x-1) (r + e^{-\vartheta_h t}(p-r))(r + e^{-\vartheta_\ell t}(q-r))dt
    \\ \label{eq:2} & = x(1-x)(2x-1) \Big( \frac{r^2}{3} + \frac{r(p-r)}{3+\vartheta_h} + \frac{r(q-r)}{3+\vartheta_\ell} + \frac{(p-r)(q-r)}{3 + \vartheta_h + \vartheta_\ell}\Big).
                      \intertext{For $\mathbb E[X_h(t) X_{\ell0}(t)^2 - X_\ell(t) X_{h0}(t)^2]$, either no
                      coalescence occurs, or colescence occurs
                      between the two $\ell$-lines ($h$-lines) in the first (second) term.
                      In this case, either no mutation occurs on both branches to the most
                      recent common ancestor, and this has type $\ell0$ ($h0$), or mutation
                      occurs on exactly on one branch, or on both branches. So,}
                      \notag\int_0^\infty & \mathbb E[X_h(t) X_{\ell0}(t)^2 - X_\ell(t) X_{h0}(t)^2] dt
    \\ & \notag = \int_0^\infty e^{-3t} x(1-x)\Big( (1-x) (r + e^{-\vartheta_\ell t}(q-r))^2 - x (r + e^{-\vartheta_h t}(p-r))^2\Big)dt
    \\ & \notag+ \int_0^\infty \int_0^t  e^{-3s} e^{-(t-s)}
    \\ \notag& \qquad \cdot
               \Big(x\Big(e^{-2\vartheta_\ell s}\mathbb E[X_{\ell0}(t-s)] +
               \underbrace{2e^{-\vartheta_\ell s}(1-e^{-\vartheta_\ell s})r\mathbb E[X_{\ell0}(t-s)]
               + (1-e^{-\vartheta_\ell s})^2 r^2(1-x)}_{= (1-e^{-\vartheta_\ell s})r(2X_{\ell0}(t) - (1-e^{-\vartheta_\ell s}) r(1-x))}\Big)
    \\ \notag & \qquad \qquad - (1-x)\Big(e^{-2\vartheta_h s}\mathbb E[X_{h0}(t-s)] +
                2e^{-\vartheta_h s}(1-e^{-\vartheta_h s})r \mathbb E[X_{h0}(t-s)]
    \\ & \notag \qquad \qquad \qquad \qquad \qquad \qquad \qquad \qquad \qquad \qquad
         \qquad \qquad \qquad + (1-e^{-\vartheta_h s})^2r^2x\Big)\Big) dsdt
    \\ \label{eq:3} & = x(1-x)\Big( (1-2x)\frac{r^2}{3} + (1-x) \Big(\frac{2r(q-r)}{3+\vartheta_\ell} +
                      \frac{(q-r)^2}{3 + 2\vartheta_\ell}\Big) - x \Big(\frac{2r(p-r)}{3+\vartheta_h}
                      + \frac{(p-r)^2}{3 + 2\vartheta_h}\Big)\Big)
    \\ \label{eq:4a} & \qquad + x(1-x) \Big[ \int_0^\infty
                       \int_s^\infty e^{-3s}e^{-(t-s)} \Big(e^{-2\vartheta_\ell s}(r + e^{-\vartheta_\ell(t-s)}(q-r))
    \\ \notag & \qquad \qquad \qquad \qquad \qquad \qquad \qquad \qquad \qquad - e^{-2\vartheta_h s}
                (r + e^{-\vartheta_h(t-s)}(p-r)) dtds
    \\ \label{eq:4b} & \qquad + r\int_0^\infty \int_s^\infty e^{-3s}e^{-(t-s)} \Big(2(1 - e^{-\vartheta_\ell s})(r + e^{-\vartheta_\ell t}(q-r)) - (1 - e^{-\vartheta_\ell s})^2r
    \\ \notag & \qquad \qquad \qquad \qquad \qquad \qquad \qquad
                - 2(1 - e^{-\vartheta_h s})(r + e^{-\vartheta_h t}(p-r)) + (1-e^{-\vartheta_h s})^2r\Big) dtds    \Big].
  \end{align}
  Now, for \eqref{eq:4a}
  \begin{align}
    \notag \int_0^\infty &
                           \int_s^\infty e^{-3s}e^{-(t-s)} \Big(e^{-2\vartheta_\ell s}(r + e^{-\vartheta_\ell(t-s)}(q-r))
                           - e^{-2\vartheta_h s}(r + e^{-\vartheta_h(t-s)}(p-r)) dtds
    \\ \notag & =     \int_0^\infty
                \int_0^\infty e^{-3s}e^{-t} \Big(e^{-2\vartheta_\ell s}(r + e^{-\vartheta_\ell t}(q-r))
                - e^{-2\vartheta_h s}(r + e^{-\vartheta_h t}(p-r)) dtds
    \\ \notag & = \frac{1}{3 + 2\vartheta_\ell} \Big( r + \frac{q-r}{1 + \vartheta_\ell}\Big)
                -\frac{1}{3 + 2\vartheta_h} \Big( r + \frac{p-r}{1 + \vartheta_h}\Big)
    \\ \label{eq:4} & = \frac{r\vartheta_\ell + q}{(3 + 2\vartheta_\ell)(1+\vartheta_\ell)}
                      - \frac{r\vartheta_h + p}{(3 + 2\vartheta_h)(1+\vartheta_h)}
  \end{align}
  and for \eqref{eq:4b},
  \begin{align}
    \notag \int_0^\infty & \int_s^\infty e^{-3s}e^{-(t-s)} \Big(2(1 - e^{-\vartheta_\ell s})(r + e^{-\vartheta_\ell t}(q-r))
                           - (1 - e^{-\vartheta_\ell s})^2r
    \\ \notag & \qquad \qquad \qquad \qquad \qquad \qquad \qquad
                - 2(1 - e^{-\vartheta_h s})(r + e^{-\vartheta_h t}(p-r)) + (1-e^{-\vartheta_h s})^2r)\Big) dtds
    \\ \notag & = \int_0^\infty \int_s^\infty e^{-2s}e^{-t} \Big(2(1 - e^{-\vartheta_\ell s})e^{-\vartheta_\ell t}(q-r)
                - e^{-2\vartheta_\ell s} r
    \\ \notag & \qquad \qquad \qquad \qquad \qquad \qquad \qquad
                - 2(1 - e^{-\vartheta_h s})e^{-\vartheta_h t}(p-r) + e^{-2\vartheta_h s} r)\Big) dtds
    \\ \notag & = \int_0^\infty e^{-3s}  \Big(2(1 - e^{-\vartheta_\ell s}) \frac{1}{1+\vartheta_\ell} e^{-\vartheta_\ell s}(q-r)
                - e^{-2\vartheta_\ell s} r
    \\ \notag & \qquad \qquad \qquad \qquad \qquad \qquad \qquad
                - 2(1 - e^{-\vartheta_h s})e^{-\vartheta_h s} \frac{1}{1 + \vartheta_h}(p-r) + e^{-2\vartheta_h s} r)\Big) dtds
    \\ \notag & = \frac{2(q-r)}{1+\vartheta_\ell} \Big(\frac{1}{3+\vartheta_\ell} - \frac{1}{3 + 2\vartheta_\ell} \Big) -
                \frac{2(p-r)}{1+\vartheta_h} \Big(\frac{1}{3+\vartheta_h} - \frac{1}{3 + 2\vartheta_h} \Big)
                + \frac{r}{3 + 2\vartheta_h} - \frac{r}{3 + 2\vartheta_\ell}
    \\ \label{eq:5}& = \frac{2\vartheta_\ell(q-r)}{(1+\vartheta_\ell)(3 + \vartheta_\ell)(3 + 2\vartheta_\ell)}
                     - \frac{2\vartheta_h(p-r)}{(1+\vartheta_h)(3 + \vartheta_h)(3 + 2\vartheta_h)}
                     + \frac{r}{3 + 2\vartheta_h} - \frac{r}{3 + 2\vartheta_\ell}.
  \end{align}
  Summing
  $\text{\eqref{eq:1}} + 2\cdot \text{\eqref{eq:2}} + 2\cdot
  \text{\eqref{eq:3}} + 2x(1-x) \cdot \text{\eqref{eq:4}} + 2x(1-x)r
  \cdot \text{\eqref{eq:5}}$ gives
  \begin{align*}
    \int_0^\infty & \mathbb E[(X_{h0}(t)X_{\ell1}(t) - X_{h1}(t)X_{\ell0}(t))(X_1(t) - X_0(t))] dt
    \\ & = x(1-x)\Big[\frac{p-r}{1+\vartheta_h} - \frac{q-r}{1+\vartheta_\ell}
    \\ & \qquad \qquad + 2\Big(\frac{r(q-r)}{3+\vartheta_\ell}
         + (1-x) \frac{(q-r)^2}{3 + 2\vartheta_\ell}\Big) - 2\Big(\frac{r(p-r)}{3+\vartheta_h}
         + x\frac{(p-r)^2}{3 + 2\vartheta_h}\Big) + (2x-1)\frac{2(p-r)(q-r)}{3 + \vartheta_\ell + \vartheta_h}
    \\ & \qquad \qquad \qquad \qquad + 2 \Big( \Big( \frac{r\vartheta_\ell + q}{(3 + 2\vartheta_\ell)(1+\vartheta_\ell)}
         + \frac{r(1-r)-r}{3 + 2\vartheta_\ell}\Big)
         - \Big(\frac{r\vartheta_h + p}{(3 + 2\vartheta_h)(1+\vartheta_h)} + \frac{r(1-r)-r}{3 + 2\vartheta_h}\Big)
    \\ & \qquad \qquad \qquad \qquad \qquad \qquad + \frac{4\vartheta_\ell
         r(q-r)}{(1+\vartheta_\ell)(3 + \vartheta_\ell)(3 + 2\vartheta_\ell)}
         - \frac{4\vartheta_h r(p-r)}{(1+\vartheta_h)(3 + \vartheta_h)(3 + 2\vartheta_h)}\Big)\Big]
    \\ & = x(1-x)\Big[\frac{p-r}{1+\vartheta_h} - \frac{q-r}{1+\vartheta_\ell} + 2\Big((1-x) \frac{(q-r)^2}{3 + 2\vartheta_\ell}
         - x\frac{(p-r)^2}{3 + 2\vartheta_h}+ (2x-1)\frac{(p-r)(q-r)}{3 + \vartheta_\ell + \vartheta_h}\Big)
    \\ & \qquad \qquad + \frac{2r(q-r)}{3+\vartheta_\ell}\Big(\underbrace{1 +  \frac{2\vartheta_\ell}{(1+\vartheta_\ell)(3 + 2\vartheta_\ell)}}_{ = \frac{(1 + 2\vartheta_\ell)(3 + \vartheta_\ell)}{(1 + \vartheta_\ell)(3 + 2\vartheta_\ell)}}
         \Big) - \frac{2r(p-r)}{3+\vartheta_h}\Big(1 +  \frac{2\vartheta_h}{(1+\vartheta_h)(3 + 2\vartheta_h)}
         \Big)
    \\ & \qquad \qquad \qquad \qquad + 2 \Big( \frac{q-r}{(3 + 2\vartheta_\ell)(1+\vartheta_\ell)}
         - \frac{p-r}{(3 + 2\vartheta_h)(1+\vartheta_h)} + r(1-r)\Big(\frac{1}{3 + 2\vartheta_\ell} - \frac{1}{3 + 2\vartheta_h}\Big)\Big)\Big]
    \\ & = x(1-x)\Big[\frac{(p-r)(1 + 2\vartheta_h)}{(1+\vartheta_h)(3 + 2\vartheta_h)}
         - \frac{(q-r)(1 + 2\vartheta_\ell)}{(1+\vartheta_\ell)(3 + 2\vartheta_\ell)}
    \\ & \qquad \qquad + 2\Big((1-x) \frac{(q-r)^2}{3 + 2\vartheta_\ell}
         - x\frac{(p-r)^2}{3 + 2\vartheta_h}+ (2x-1)\frac{(p-r)(q-r)}{3 + \vartheta_\ell + \vartheta_h}\Big)
    \\ & \qquad \qquad     \qquad \qquad
         + \frac{2r(q-r)(1 + 2\vartheta_\ell)}{(1+\vartheta_\ell)(3 + 2\vartheta_\ell)}
         - \frac{2r(p-r)(1+2\vartheta_h)}{(1+\vartheta_h)(3+2\vartheta_h)} + 2r(1-r) \Big(\frac{1}{3 + 2\vartheta_\ell} - \frac{1}{3 + 2\vartheta_h}\Big)\Big]
    \\ & = x(1-x) \Big[\frac{(2r-1)(q-r)(1 + 2\vartheta_\ell)}{(1+\vartheta_\ell)(3 + 2\vartheta_\ell)}
         - \frac{(2r-1)(p-r)(1+2\vartheta_h)}{(1+\vartheta_h)(3+2\vartheta_h)}
    \\ & \quad + 2\Big((1-x) \frac{(q-r)^2}{3 + 2\vartheta_\ell}
         - x\frac{(p-r)^2}{3 + 2\vartheta_h}+ (2x-1)\frac{(p-r)(q-r)}{3 + \vartheta_\ell + \vartheta_h} + r(1-r) \Big(\frac{1}{3 + 2\vartheta_\ell} - \frac{1}{3 + 2\vartheta_h}\Big)\Big)\Big]
  \end{align*}
  which together with \eqref{eq:T3int} shows \eqref{eq:T3}.
\end{proof}

\subsubsection*{Acknowledgments}
FB was supported by the DFG priority program SPP 2141 through grant
Ba-5529/1-1. PP was supported by the DFG priority program SPP 1590
through grant Pf-672/8-1.

\bibliographystyle{chicago}
\bibliography{secondorderevolution}

\end{document}